\documentclass[12pt]{amsart}

\usepackage{amsthm}
\usepackage{amsmath}
\usepackage{amsfonts,amssymb,latexsym}
\usepackage[english]{babel}
\usepackage{epsfig,epsf}
\usepackage{color}
\newtheorem{thm}{Theorem}
\newtheorem{lemma}{Lemma}

\newtheorem{remark}[thm]{Remark}

\begin{document}

\title[Equation $A!B!=C!$]{Explicit bounds for the diophantine equation $A!B!=C!$}
\author{Laurent Habsieger}
\address{Universit\'e de Lyon, CNRS UMR 5208, Universit\' e Claude Bernard Lyon 1,
Institut Camille Jordan, 43 boulevard du 11 novembre 1918, 69622 Villeurbanne Cedex, France \bigskip}
\email{habsieger@math.univ-lyon1.fr}

\begin{abstract} A nontrivial solution of the equation $A!B!=C!$ is a triple of positive integers $(A,B,C)$ with $A\le B\le C-2$. It is conjectured that the only nontrivial solution is $(6,7,10)$,
and this conjecture has been checked up to $C=10^6$. Several estimates on the relative size of the parameters are known, such as the one given by Erd\"os $C-B\le5\log\log C$, 
or the one given by Bhat and Ramachandra $C-B\le(1/\log2+o(1))\log\log C$. We check the conjecture for $B\le10^{3000}$ 
and give better explicit bounds such as $C-B\le\frac{\log\log(B+1)}{\log2}-0.8803$.
\end{abstract}

\maketitle

\section{Introduction}

Many authors \cite{G} considered the diophantine equation
\begin{equation}\label{E}
n!=\prod_{i=1}^r a_i!
\end{equation}
in the integers $r,a_1,\dots,a_r$,  with $r\ge2$ and $a_1\ge\dots \ge a_r\ge2$. A trivial solution is given by $a_1=n-1$ and $n=\prod_{i=2}^r a_i!$. Hickerson conjectured that the only
non-trivial solutions are $9!=7!3!3!2!$, $10!=7!6!=7!5!3!$ and $16!=14!5!2!$. He checked it for $n\le410$, which was improved to $18160$ by Shallit and Easter (see \cite{G}).
Sur\'anyi also conjectured the case $r=2$ (see \cite{E75}) and this was verified up to  $n=10^6$ by Caldwell \cite{C}.

Luca \cite{L} proved there are finitely many non-trivial solutions to \eqref{E}, 
assuming the $abc$- conjecture. Erd\"os \cite{E75} showed that, if the largest prime number of $n(n+1)$ is greater than
$4\log n$ for any positive integer $n$, then there are only finitely many nontrivial solutions to  \eqref{E}.

From now on, we shall focus on the case $r=2$, i. e. the equation
\begin{equation}\label{*}
A!B!=C!\,,
\end{equation}
which has been studied by Caldwell \cite{C} for $C\le 10^6$. Erd\"os \cite{E93} proved 
that $C-B\le5\log\log C$ for $C$ sufficiently large, and noted that it would be nice to obtain a bound of the form
$C-B=o(\log\log C)$. His result was improved by Bhat and Ramachandra \cite{BR}, who showed that $C-B\le(1/\log2+o(1))\log\log C$. 
Hajdu, Papp and Szak\'acs \cite{HPS} recently proved that non-trivial solutions different from $10!=7!6!$ satisfy to $C<5(B-A)$ and $B-A\ge10^6$.
The aim of this paper is to get better explicit inequalities.

Let $a\ge2$ be an integer. Let $s_a$ denote the sum of the digits of an integer written in the basis $a$. 
When $p$ is a prime, Legendre's formula gives the exponent of $p$ in $n!$:
$$v_p(n!)=\frac{n-s_p(n)}{p-1}\,.$$
When we apply this formula to \eqref{*}, we find $A-v_p(A)+B-v_p(B) =C-v_p(C)$. Since $v_p(C)\ge1$ and $v_p(n)\le \frac{(p-1)\log(n+1)}{\log p}$ (see Lemma 1 below), we obtain 
\begin{equation}\label{lowerC}
C\ge A+B+1-\frac{\log(A+1)}{\log2}-\frac{\log(B+1)}{\log2}\,.
\end{equation}
Since $\log C!= \log A! + \log B!$, the condition \eqref{lowerC} implies that $A$ is much smaller that $B$. We shall make this assertion explicit by proving the following theorem.

\begin{thm}\label{main1} Let $(A,B,C)\neq(6,7,10)$ be a nontrivial solutions triple of \eqref{*}.
For any real number $t>-1-\frac{1+2\log\log2}{\log2}=-1.3851\dots$ we have
$$A\le \frac{\log(B+1)}{\log2}+\frac{2\log\log(B+1)}{\log2}+t$$
when $B$ is sufficiently large. Moreover we have
$$A\le \frac{\log(B+1)}{\log2}+\frac{2\log\log(B+1)}{\log2}+2.1221\,.$$
\end{thm}

We can slightly improve on Bhat and Ramachandra's result \cite{BR}.
\begin{thm}\label{main2} Let $(A,B,C)\neq(6,7,10)$ be a nontrivial solution triple of \eqref{*}.
For any real number $u>-\frac{1+\log\log2}{\log2}=-0.9139\dots$, we have
$$C-B\le\frac{\log\log(B+1)}{\log2}+u$$
when $B$ is sufficiently large. Moreover we have
$$C-B\le\frac{\log\log(B+1)}{\log2}+1.819\,.$$
\end{thm}

We also deduce a better explicit estimate than $B-A>C/5$ given by Hajdu, Papp and Szak\'acs \cite{HPS}.
\begin{thm}\label{main3} Let $(A,B,C)\neq(6,7,10)$ be a nontrivial solution triple of \eqref{*}.
For any real number $v<1+\frac{2+3\log\log2}{\log2}=2.299\dots$, we have
$$B-A>C-\frac{\log(C+1)}{\log2}-\frac{3\log\log(C+1)}{\log2}+v$$
when $B$ is sufficiently large. Moreover we have
$$B-A>C-\frac{\log(C+1)}{\log2}-\frac{3\log\log(C+1)}{\log2}-3.9411\,.$$
\end{thm}

All these general estimates used the fact that $B\ge 10^6$ for nontrivial solutions triple distinct from $(6,7,10)$. We use these estimates to improve  
both on the range of validity of Sur\'anyi's conjecture and the estimates given before.

\begin{thm}\label{main4} Let $(A,B,C)\neq(6,7,10)$ be a nontrivial solution triple of \eqref{*}. Then we have $B\ge 10^{3000}$ and
$$\begin{aligned}
A&\le \frac{\log(B+1)}{\log2}+\frac{2\log\log(B+1)}{\log2}-1.3479\,,\\
C-B&\le\frac{\log\log(B+1)}{\log2}-0.8803\,,\\
B-A&>C-\frac{\log(C+1)}{\log2}-\frac{3\log\log(C+1)}{\log2}+2.2282\,.\\
\end{aligned}$$
\end{thm}

\begin{remark} Caldwell's result $C\ge 10^6$ concerning Sur\'anyi's conjecture is extended to the much larger region $C\ge 10^{3000}$.
\end{remark}

We first establish useful general properties for the sum of digits and for the $\Gamma$ function in the next section. In section 3, we prove a key lemma that studies the asymptotic 
behaviour of $\log C!-\log A! - \log B!$ under the condition \eqref{lowerC}, for $A=\frac{\log(B+1)}{\log2}+\frac{2\log\log(B+1)}{\log2}+t$.
We deduce Theorems \ref{main1}-\ref{main3} in section 4. In section 5 we use these results to prove Theorem \ref{main4}, hence also to
check Sur\'anyi's conjecture further, and to improve on the results of the preceeding section.
We end this paper with a few remarks on possible ways to get better results.

\section{General properties of $s_a$ and $\Gamma$}

We first give a tight upper bound for the sum of the digits function.

\begin{lemma}\label{bounds_a}
Let $a\ge2$ be an integer. For any nonnegative integer $n$, we have the upper bound
$$s_a(n)\le \frac{(a-1)\log(n+1)}{\log a}\,.$$
\end{lemma}

\begin{proof} Let $n$ be a nonnegative integer. Write $s_a(n)=(a-1)b+r$, where $b$ is a nonnegative integer and $0\le r\le a-2$. We have
$$ n \ge \sum_{i=0}^{b-1}(a-1)a^i + ra^b=(r+1)a^b-1\,.$$
The function $x\to  x-(a-1)\frac{\log(x+1)}{\log a}$ is convex and vanishes at $x=0$ and $x=a-1$. Therefore this function is nonpositive on the interval
$\lbrack0,a-1\rbrack$. We thus get
$$s_a(n)= (a-1)b+r\le (a-1)\frac{\log(a^b)}{\log a}+(a-1)\frac{\log(r+1)}{\log a}\le \frac{(a-1)\log(n+1)}{\log a}\,.$$
\end{proof}

Put $\Psi(z)=\Gamma'(z)/\Gamma(z)$. Let $\gamma$ denote Euler's constant. We recall the formulas (see \cite{EMOT}, p. 15)
\begin{equation}\label{Psi}
\begin{aligned}
\Psi(z)&=-\gamma+\sum_{k=0}^{\infty} \left(\frac{1}{k+1}-\frac{1}{z+k}\right)\\
\Psi'(z)&=\sum_{k=0}^{\infty}\frac{1}{(z+k)^2}\,,
\end{aligned}
\end{equation}
and Binet's second expression for $\log\Gamma$ (see \cite{EMOT}, p. 22)
\begin{equation}\label{Binet}
\log\Gamma(x)=\left(x-\frac{1}{2}\right)\log x-x+\frac{\log(2\pi)}{2}+2\int_0^{\infty} \frac{\arctan(t/x)}{e^{2\pi t}-1} dt\,.
\end{equation}
From the bounds $0\le\arctan(t/x)\le t/x$ and from \eqref{Binet}, we get the well-known explicit Sirtling's formula
\begin{equation}\label{Stirling}
0\le \log\Gamma(x)-x(\log x-1)-\frac{\log(2\pi/x)}{2}\le \frac{1}{12x}\,.
\end{equation}
Derivating \eqref{Binet} also leads to the formula
$$\Psi(x)=\log x-\frac{1}{2x}-\int_0^{\infty} \frac{2t}{(t^2+x^2)(e^{2\pi t}-1)} dt$$
and the bounds $0\le 1/(t^2+x^2)\le1/x^2$ give the estimates
\begin{equation}\label{Stirling2}
-\frac{1}{12x^2}\le \Psi(x)-\log x+\frac{1}{2x}\le0\,.
\end{equation}

\section{The key lemma}

Let us define
$$R(A,B)=\log\Gamma\left(A+B+2-\frac{\log(A+1)}{\log2}-\frac{\log(B+1)}{\log2}
\right)-\log\Gamma\left(A+1\right)-\log\Gamma\left(B+1\right)\,,$$
and let us put 
$$A_t=\frac{\log(B+1)}{\log2}+\frac{2\log\log(B+1)}{\log2}+t$$
for any real number $t$.

\begin{lemma} \label{key} Let $t$ be a real number, $t> -1-\frac{1+2\log\log2}{\log2}=-1.3851\dots$ . There exists a function $C(t,B+1)$ such that
$$R(A_t,B) \ge C(t,B+1)\log(B+1)\,,$$
with
$$\lim_{B\to+\infty} C(t,B+1)=t+1+\frac{1+2\log\log2}{\log2}>0\,.$$
Moreover we can have  $C(2.1221,B+1)>0$ for $B\ge 10^6$.
\end{lemma}

\begin{proof} For $B\ge2$, we can have
$$\begin{aligned}
\log\left( A_t+1\right)&=\log\left(\frac{\log(B+1)}{\log2}+\frac{2\log\log(B+1)}{\log2}+t+1\right)\\
&\le \log\log(B+1)-\log\log2+\frac{2\log\log(B+1)+(t+1)\log2}{\log(B+1)}
\end{aligned}$$
and therefore
$$ \displaylines{
A_t+B+2-\frac{\log(A_t+1)}{\log2}-\frac{\log(B+1)}{\log2}=B+t+2+\frac{2\log\log(B+1)}{\log2}-\frac{\log(A_t+1)}{\log2}
\hfill\cr\hfill
\ge B+t+2+\frac{\log\log2}{\log2}+\frac{\log\log(B+1)}{\log2}-\frac{2\log\log(B+1)+(t+1)\log2}{\log2 \log(B+1)}
>B+1\,,}$$
for $B\ge35$. We thus get from \eqref{Psi} and \eqref{Stirling2}, for $B\ge35$:
$$\begin{aligned}
\log\Gamma&\left(A_t+B+2-\frac{\log(A_t+1)}{\log2}-\frac{\log(B+1)}{\log2}\right)-\log\Gamma\left(B+1\right)\\
&\ge \left(\frac{\log\log(B+1)}{\log2}+t+1+\frac{\log\log2}{\log2}-\frac{2\log\log(B+1)+(t+1)\log2}{\log2 \log(B+1)}\right) \Psi(B+1)\\
&\ge\left(\frac{\log\log(B+1)}{\log2}+t+1+\frac{\log\log2}{\log2}-\frac{2\log\log(B+1)+(t+1)\log2}{\log2 \log(B+1)}\right) \\
&\hskip 7 cm \times\left( \log(B+1)-\frac{1}{2(B+1)}-\frac{1}{12(B+1)^2}\right)\,.\\
&=  \left(\frac{ \log\log(B+1)}{\log2}+t+1+\frac{\log\log2}{\log2}+\varphi_1(t,B+1)\right) \log(B+1)
\end{aligned}$$
with
$$\displaylines{\varphi_1(t,x)=-\frac{2\log\log x+(t+1)\log2}{\log2 \log x}
\hfill\cr\hfill
-\frac{1}{\log x}\left(\frac{1}{2x}+\frac{1}{12x^2}\right)\left(\frac{\log\log x}{\log2}+t+1+\frac{\log\log2}{\log2}-\frac{2\log\log x+(t+1)\log2}{\log2 \log x}\right)\,.}$$

Stirling's formula \eqref{Stirling} gives
$$\log\Gamma(x)\le x(\log x-1)+\frac{\log(2\pi/x)}{2}+ \frac{1}{12x}\le x(\log x-1)$$
for $x\ge 6.448$, from which we obtain
$$\displaylines{\log\Gamma\left(A_t+1\right)\le \left(\frac{\log(B+1)}{\log2}+\frac{2\log\log(B+1)}{\log2}+t+1\right)
\hfill\cr\hfill
\times\left(\log\log(B+1)-1-\log\log2+\frac{2\log\log(B+1)+(t+1)\log2}{\log(B+1)}\right)\,,}$$
when $A_t\ge 6.448$. Since $A_t>A_{-1-\frac{1+2\log\log2}{\log2}}\ge 6.448$ for $B\ge 23$, we get
$$\log\Gamma\left(A_t+1\right)\le \left(\frac{ \log\log(B+1)}{\log2}-\frac{1+\log\log2}{\log2}+\varphi_2(t,B+1)\right) \log(B+1)$$
for $B\ge23$, with 
$$\varphi_2(t,x)=\frac{2\log\log x+(t+1)\log2}{\log2\log x}\left( \log\log x-\log\log2+\frac{2\log\log x+(t+1)\log2}{\log x} \right)\,.$$
We deduce 
$$R(A_t,B)
\ge \left(t+1+\frac{1+2\log\log2}{\log2} +\varphi_1(t,B+1)-\varphi_2(t,B+1)\right)\log(B+1)\,,$$
for $B\ge35$, and we put $C(t,B+1)= t+1+\frac{1+2\log\log2}{\log2} +\varphi_1(t,B+1)-\varphi_2(t,B+1)$.
Note that the functions $\varphi_1(t,x)$ and $\varphi_2(t,x)$ tend to $0$ when $x$ goes to infinity, which proves the first part of the lemma.

For $t\ge-1$ and $x\ge10^6$, we have
$$\begin{aligned}
-C(t,x)\le &\frac{2\log\log x+(t+1)\log2}{\log2\log x}\left( \log\log x+1-\log\log2+\frac{2\log\log x+(t+1)\log2}{\log x} \right)\\
&+\left(\frac{1}{2x}+\frac{1}{12x^2}\right)\left( 0.2634+0.0672(t+1)\right)+t+1+\frac{1+2\log\log2}{\log2}\,,
\end{aligned}$$
a decreasing function of $x$. We thus deduce $C(2.1221,10^6)>0.000016$, which completes the proof .
\end{proof}

\section{Proof of the first three theorems}

\subsection{Proof of Theorem \ref{main1}}

\begin{lemma} If  $(A,B,C)$ is a solution of \eqref{*}, then $R(A,B)\le0$. The function $R$ is an increasing function of $A$ for $1\le A\le B$.
\end{lemma}

\begin{proof} The first claim follows directly from \eqref{lowerC}: $R(A,B)\le \log C!-\log A!-\log B!=0$. We compute
$$\displaylines{\frac{\partial^2R}{\partial A\partial B}(A,B)
\hfill\cr\hfill
=\left( 1-\frac{1}{(A+1)\log2}\right) \left( 1-\frac{1}{(B+1)\log2}\right)  \Psi'\left(A+B+2-\frac{\log(A+1)}{\log2}-\frac{\log(B+1)}{\log2}\right)\,.}$$
From \eqref{Psi}  we get $\frac{\partial^2R}{\partial A\partial B}(A,B)\ge0$ for $1\le A\le B$. We use \eqref{Psi} to deduce
$$\begin{aligned}
\frac{\partial R}{\partial A}(A,B)&\ge \frac{\partial R}{\partial A}(A,A)=\left( 1-\frac{1}{(A+1)\log2}\right) \Psi\left(2A+2-2\frac{\log(A+1)}{\log2}\right)-\Psi(A+1)\\
&=\frac{\gamma}{(A+1)\log2}+\sum_{k=0}^{\infty} \left(\frac{1}{k+A+1}-\frac{1-\frac{1}{(A+1)\log2}}{k+2A+2-2\frac{\log(A+1)}{\log2}}\right)\\
&=\frac{\gamma}{(A+1)\log2}+\sum_{k=0}^{\infty}\frac{ \frac{k}{(A+1)\log2}+A+1-2\frac{\log(A+1)}{\log2}+\frac{1}{\log2}}{(k+A+1)(k+2A+2-2\frac{\log(A+1)}{\log2})}>0
\end{aligned}$$
when $A+1\ge\max\left(2\frac{\log(A+1)}{\log2}-\frac{1}{\log2},\frac{\log(A+1)}{\log2}\right) \ge0$, which is true for $A\ge1$.
\end{proof}

Thus we only need to find $\bar A$ such that $R(\bar A,B)>0$ to get a bound $A<\bar A$. For $t>-1-\frac{1+2\log\log2}{\log2}$ we have $R(A_t,B)>0$ for $B$ large enough by Lemma
\ref{key}, which gives the first part of Theorem \ref{main1}. 
Hajdu, Papp and Szak\'acs \cite{HPS} proved $B-A\ge10^6$, which ensures us that $B\ge10^6$. We can therefore deduce the second part of the theorem from the inequality
$C(2.1221,B+1)>0$, also given in Lemma \ref{key}.

\subsection{Proof of Theorem \ref{main2}}

Note that
$$\log A!=\log\frac{C!}{B!}\ge (C-B)\log(B+1)\,.$$
For $A\le A_t$, we have showed in the proof of Lemma \ref{key} that
$$\log A!\le \log\Gamma(A_t+1)\le \left(\frac{ \log\log(B+1)}{\log2}-\frac{1+\log\log2}{\log2}+\varphi_2(t,B+1)\right) \log(B+1)\,.$$
Therefore
$$C-B\le \frac{ \log\log(B+1)}{\log2}-\frac{1+\log\log2}{\log2}+\varphi_2(t,B+1)\,,$$
thus proving the first part of the theorem, since $\varphi_2(t,x)$ tend to $0$ when $x$ goes to infinity.

Each monomial term $(\log\log x)^n(\log x)^{-m}$ defining $\varphi_2$ is a positive decreasing function of $x$ for $t\ge-1$ and $x\ge 10^6$. We find
$-\frac{1+\log\log2}{\log2}+\varphi_2(2.1221,10^6)<1.819$ and the theorem follows, as in the previous subsection.

\subsection{Proof of Theorem \ref{main3}}

We write $B-A=C-A-(C-B)$ and we use Theorems  \ref{main1} and  \ref{main2} to get
$$B-A\ge C-\frac{\log(B+1)}{\log2}-\frac{3\log\log(B+1)}{\log2}-3.9411\,.$$
The second part of the theorem follows, and the first part is straightforward.

\section{The proof of Theorem \ref{main4}}

Theorems \ref{main2} and \ref{main3} show that both $A$ and $C-B$ are small with respect to $B$. 
Let us put $k=C-B$ to simplify the statements.

\begin{lemma}\label{valb} Let $(A,B,C)$ a be a nontrivial solutions triple of \eqref{*}.
For $k=C-B\in\{2,3,\dots,20\}$, we have $B=B_k(A):=\lceil (A!)^{1/k}-(k+1)/2 \rceil$.
\end{lemma}

\begin{proof} We have 
$$A! =\prod_{i=1}^k(B+i)=\prod_{i=1}^k \sqrt{(B+i)(B+k+1-i)}< \left(B+\frac{k+1}{2}\right)^k\,,$$
which shows that $B> (A!)^{1/k}-(k+1)/2$.

We used MAPLE to check that the polynomial $\prod_{i=1}^k(B+i) -(B+(k-1)/2)^k$ is a polynomial in $B-1$ with nonnegative coefficients and with a positive value at $B=1$, for $2\le k\le 12$.
This implies that $B< (A!)^{1/k}-(k-1)/2$, and the lemma follows.
\end{proof}

We checked that the inequality $A! =\prod_{i=1}^k(B_k(A)+i)$ never occurred for $A\le 10000$ and $2\le k\le12$ using MAPLE; we 
asked for a $40000$-digits precision (enough to write all the digits of $A!$), and this required about twenty-eight hours of computations. 

For $B\le 10^{1000}$, Theorems \ref{main2} and \ref{main3} give $A\le 3346$ and $k\le 12$, so that the equation \eqref{*} has no solution for $10^6\le B\le 10^{1000}$.
We can get better inequalities in these theorems, using $B\ge10^{1000}$. Computing  $C(-1.2979,10^{1000})$ and $\varphi_2(1.2979,10^{1000})$ leads to
$$\begin{aligned}
A&\le \frac{\log(B+1)}{\log2}+\frac{2\log\log(B+1)}{\log2}-1.2979\,,\\
C-B&\le\frac{\log\log(B+1)}{\log2}-0.8362\,.\end{aligned}$$
For  $10^{1000}\le B\le 10^{3000}$, we thus obtain $A\le 9993$ and $k\le 11$, and  the equation \eqref{*} has no solution on this interval. 
Computing $C(-1.3479,10^{3000})$ and $\varphi_2(1.3479,10^{3000})$ gives the inequalities from Theorem \ref{main4}.

\section{Concluding remarks}

Our method is based on two informations: an arithmetical information obtained by considering the dyadic valuation of the factorials, 
and an asymptotic information obtained from Stirling's formula. In order to improve on the orders of magnitude of our estimates, one should get more arithmetical information.
First, we applied the estimate from Lemma \ref{bounds_a} both for $A!$ and for $B!$, and it is quite uncommon that this estimate can be sharp in both cases.
Second, we did not use any property of the $p$-adic valuations for $p\ge3$, and any useful information could lead to improvements.

The algorithm we used to check that $A!B_k(A)!\neq(B_k(A)+k)!$ is rather basic. A smarter one should lead to an even much larger bound than ours.

\medskip

\noindent MSC2010: 11D85 11B65 11D41 11N64

\end{document}